\documentclass{article}

\usepackage{graphicx}

\usepackage{amsmath,amsfonts,amssymb,latexsym,epsfig}
\usepackage{mathrsfs}
\usepackage{verbatim}
\usepackage{latexsym}
\usepackage{amsthm}
\usepackage{amssymb}
\usepackage{subfig}
\usepackage{graphics}
\usepackage{amsbsy}
\usepackage{fullpage}
\usepackage{enumerate}
\usepackage{times}
\usepackage{multirow}
\usepackage{soul}
\usepackage[normalem]{ulem}

\newcommand{\R}{\mathbb R}

\newtheorem{theorem}{Theorem}[section]
\newtheorem{lemma}[theorem]{Lemma}

\newtheorem{remark}[theorem]{Remark}

\numberwithin{equation}{section}
\usepackage{amscd}
\usepackage{tikz}
\usepackage{xcolor}
\usepackage[english]{babel}
\usepackage[latin1]{inputenc}
\usepackage{times}
\usepackage[T1]{fontenc}
\usepackage{graphicx}
\usepackage{amsmath,amssymb}
\usepackage{slidesec}
\usepackage{verbatim}
\usepackage{url}
\usepackage{epsf}
\usepackage{amsmath}
\usepackage{graphics}
\usepackage{amsfonts}
\usepackage{amsbsy}
\usepackage{lscape}
\usepackage{enumerate}
\usepackage{amsthm}
\usepackage{amssymb}
\usepackage{exscale}
\usepackage{algorithm}
\usepackage{algorithmic}

\newcommand{\T}{{\mathbb T}}

\newcommand{\modk}[1]{\textcolor{black}{#1}\index {#1}}

\newcommand{\modkzz}[1]{\textcolor{black}{#1}\index {#1}}

\newcommand{\wsigma}{\widetilde{\sigma}}
\newcommand{\wpoo}{\widetilde{p}_{11}}
\newcommand{\wpot}{\widetilde{p}_{12}}
\newcommand{\wptt}{\widetilde{p}_{22}}

\begin{document}
\title{The connections between Lyapunov functions for some optimization algorithms and differential equations.}
\author{J. M. Sanz Serna $^{1}$ \and Konstantinos C. Zygalakis$^{2}$}

\maketitle

\begin{abstract}
\modk{In this manuscript we study the properties of a family
of a
second order differential equations with damping, its
discretizations
and their connections with accelerated optimization
algorithms for
$m$-strongly convex and $L$-smooth functions. In
particular, using
the Linear Matrix Inequality {\color{black}(LMI)} framework
developed by  \emph{Fazlyab et. al. $(2018)$},
we derive analytically a (discrete) Lyapunov function for a
two-parameter family of Nesterov optimization methods,
which allows
for \modkzz{the} complete characterization of their
convergence rate. {\color{black} In the appropriate limit,  this
family
of methods may be seen as a discretization of a family of
second
order ordinary differential equations for which we construct
(continuous) Lyapunov functions by means of the LMI
framework. The continuous Lyapunov functions may
alternatively be obtained by studying the limiting behaviour
of their discrete counterparts.} % end of blue
Finally, we show  that the majority of typical discretizations
of the {\color{black} of the family of ODEs}, such as the Heavy
ball method,  do not possess {\color{black} Lyapunov functions with properties similar to those of the Lyapunov function constructed here for the Nesterov method.}}

\end{abstract}

%\modk{In this manuscript we study the properties of a family of second order differential equation with damping from the perspective of Lyapunov functions using the Integral Quadratic Constraints framework developed in \cite{FRMP} }

\section{Introduction}
This paper studies Lyapunov functions for  differential equations with damping, their  discretizations, and  optimization algorithms.

The simplest algorithm for solving
\[
\min_{x \in \R^{d}} f(x)
\]
is the gradient descent (GD) method
\[
x_{k+1}=x_{k}-\alpha_k \nabla f(x_{k}),
\]
which is of course the result of applying Euler's rule, with step-size \(\alpha_k>0\), to the gradient system
\[
\frac{dx}{dt}=-\nabla f(x),  \qquad x(0)=x_{0}.
\]
The value of \(f\) decreases along solutions \(x(t)\) of  this system and, correspondingly, it may be hoped that, for GD,
\(f(x_{k+1})\leq f(x_k)\) for sufficiently small \(\alpha_k\). In fact, that is the case for \(\alpha_k<2/L\) if
 \(f\) is \(L\)-smooth, i.e.\ if \(\nabla f(x)\) is \(L\)-Lipschitz continuous. In this paper we are mainly interested in problems where \(f\) belongs the set
\( \mathcal{F}_{m,L}\)
of \(m\)-strongly convex and \(L\)-smooth functions, a class that plays an important role in optimization
\cite{N14}. For \(f\) in this class  and the constant step-size
$\alpha=2/(m+L)$, GD has a bound \cite[Theorem 2.1.15]{N14}
\begin{equation}\label{eq:GDbound}
f(x_k) - f(x^\star) \leq \frac{L}{2}\left(\frac{1-1/\kappa}{1+1/\kappa}\right)^{2k} \| x_0-x^\star\|^2,
\end{equation}
where \(x^\star\) is the (unique) minimizer of \(f\) and \(\kappa = L/m\geq 1\) is the condition number of \(f\).

The {\color{black} \(1-\mathcal{O}(1/\kappa)\) rate of} decay in \(f\) in the preceding bound is unsatisfactory because in many
 applications of interest one has $\kappa \gg 1$. It is possible to improve on GD by resorting to \emph{accelerated} algorithms {\color{black} with rates \(1-\mathcal{O}(1/\sqrt{\kappa})\);} for instance, for the method
\begin{subequations} \label{eq:nest}
\begin{align}
x_{k+1} &=y_{k}-\frac{1}{L} \nabla f(y_{k}), \\
y_{k} &=x_{k}+\frac{1-\sqrt{1/\kappa}}{1+\sqrt{1/\kappa}}(x_{k}-x_{k-1}),
\end{align}
\end{subequations}
introduced by Nesterov, it may be shown \cite[Theorem 2.2.3]{N14} that, if \(y_0= x_0\),
\begin{equation}\label{eq:boundfrombook}
f(x_{k}) -f(x^\star) \leq \left(1-\sqrt{1/\kappa}\right)^{k} \Big( f(x_0)-f(x^\star)+ \frac{m}{2} \|x_0-x^\star\|^2\Big).
\end{equation}
 The factor \(1-\sqrt{1/\kappa}\) here is close to the optimal possible
 factor \((1-\sqrt{1/\kappa})^2/(1+\sqrt{1/\kappa})^2\) one can
 achieve for minimization algorithms when $f \in \mathcal{F}_{m,L}$
 \cite[Theorem 2.1.13]{N14}. The algorithm \eqref{eq:nest} is also
related to ODEs, because it may be seen as a discretization of
 of the Polyak damped oscillator equation \cite{P64}
\begin{equation} \label{eq:Pol_ODE}
\ddot{x}+2\sqrt{m}\dot{x}+\nabla f(x)=0,
\end{equation}
whose solutions \(x(t)\) approach \(x^\star\) as \(t\rightarrow \infty\) if \(f\) is \(m\)-strongly convex \cite[Proposition 3]{WRJ16}.

In  recent years,  there has been a revived interest, beginning with
\cite{SBC16}, in the connections between differential equations and
optimization algorithms (\modk{see also \cite{SRB17}). In particular,
there has been {\color{black} several papers (see e.g.\ \cite{WWJ16,KBB15})} %end of colour blue
 that proposed accelerated algorithms, both in  Euclidean and non-
Euclidean geometry,  based on discretizations of second order
dissipative ODEs. The structure of these ODEs and the fact that
they can been viewed as describing Hamiltonian systems with
dissipation, led to a number of research works that tried to construct or explain optimization algorithms using concepts such as shadowing \cite{OL19}, {\color{black} symplecticity} % end of color blue
\cite{BJW18,BDF19,MJ19,MJ20,SDS19}, discrete gradients \cite{ERRS18},  and backward error analysis \cite{FJV20}.}

\modk{A common feature of the analysis presented  in many of the papers mentioned above was the construction of a discrete Lyapunov function that was used in order to deduce the convergence rate of the underlying algorithm.}  In \cite{WRJ16} a general analysis of optimization methods based on the derivation of Lyapunov functions that mimic ODE Lyapunov functions  was carried out; that paper presents a Lyapunov function for \eqref{eq:Pol_ODE}. A Lyapunov function for \eqref{eq:nest} may be \modk{seen in \cite{LO20}, where it was also used to study stochastic versions of the algorithm. {\color{black} The paper \cite{SDJ18}, among other contributions, constructs a Lyapunov function for a one-parameter family of optimization algorithms that includes \eqref{eq:nest} as a particular case.} %end of color blue
Outside the field of optimization, Lyapunov functions are important in establishing ergodicity of random dynamical systems \cite{SS99}, as well as ergodicity of Markov Chain Monte Carlo algorithms, see for example \cite{MT93,BRSS17}}. The construction of Lyapunov functions for optimization algorithms from the perspective of control theory was  the subject of study  in  \cite{FRMP}. The authors extend the work in  \cite{LRP16} and derive  Linear Matrix Inequalities (LMIs) that guarantee the existence of   suitable Lyapunov functions that may be used to establish the convergence rate of the algorithm under study. In addition, \cite{FRMP} develops
an LMI framework to construct Lyapunov functions for systems of ODEs. Typically, the LMIs that appear in this context have {\color{black} been solved numerically in the literature}.

In this work,
\begin{enumerate}
\item For \(f\in \mathcal{F}_{m,L}\), we use the LMI framework from \cite{FRMP} to derive \emph{analytically} Lyapunov functions for a two-parameter family of Nesterov optimization methods (see \eqref{eq:nest1} below); {\color{black}  this family includes the one-parameter family of algorithms in \cite{SDJ18}.} %color blue has ended
     In this way we find, as a function  of the {\color{black} two parameters in \eqref{eq:nest1},} %color read ended
       a convergence rate for the  methods in the family. It turns out that the best convergence rate is achieved when the parameters are chosen as in \eqref{eq:nest}. {\color{black} The relation between the Lyapunov function
       constructed in the present work and its counterpart in \cite{SDJ18} is discussed
       in Remark~\ref{rem:comparisonLyapunov}.} % end of color blue

\item By taking an appropriate limit of the parameters  {\color{black} as in e.g.\ \cite{Scieur:2016,BJW18,SDJ18,BDF19,MJ19,MJ20,SDS19,FJV20} } the optimization algorithms in the family may be seen as discretizations of second-order ODEs of the form
\begin{equation} \label{eq:Pol_ODE1}
\ddot{x}+\bar{b} \sqrt{m} \dot{x}+\nabla f(x)=0,
\end{equation}
where \(\bar b>0\) is a friction parameter. We obtain analytically Lyapunov functions for \eqref{eq:Pol_ODE1}
and determine, as a function of \(\bar b\), a convergence rate of \(f\) to \(f(x^\star)\) along  solutions \(x(t)\).
 We prove that the value $\bar{b}=2$ in the Polyak ODE \
 \eqref{eq:Pol_ODE} yields the \emph{optimal convergence
 rate if \( f\) is \(m\)-strongly convex}. \modkzz{Additionally we
 show that if one is to take explicitly into account the value of
 $L$ into this calculation, the optimal value of $\bar{b}$
becomes strictly larger than $2$ and yields slightly better convergence rates}.
\item  We show that, in the limit where the optimization algorithms approximate the ODEs, the discrete Lyapunov functions converge to the ODE Lyapunov function. \modkzz{Using this correspondence
 we show}, by means of the Heavy Ball method \modk{\cite{P64}} and other examples, that typically, optimization algorithms that are discretizations of \eqref{eq:Pol_ODE1} {\color{black} do not possess discrete Lyapunov functions that mimic the Lyapunov function of the differential equation in item 2 above and lead to acceleration.}
    This emphasizes the well-known fact that, when designing optimization methods, it is not sufficient to ensure that the algorithm may be seen as a consistent discretization of a well-behaved ODE. Unfortunately, discretizations do not necessarily inherit the good long-time properties of the differential equation, \modk{as seen for example in the case of discretization of gradient flows \cite{SSKZ20}, and Hamiltonian problems \cite{sanz2018numerical}.}
\end{enumerate}

The rest of the paper is organized as follows. In Section~\ref{sec:prelim} we briefly review the approach in \cite{FRMP} that provides a basis for our constructions. In Section~\ref{sec:nest_anal} we find analytically Lyapunov functions/rates of convergence for a two-parameter family of optimization methods that contains \eqref{eq:nest} as a particular case. Section~\ref{sec:ode} analyzes the ODE \eqref{eq:Pol_ODE1} and Section~\ref{sec:connect} studies the connection between the discrete and continuous Lyapunov functions. The Heavy Ball method and other methods that do not possess suitable Lyapunov functions are discussed in Section~\ref{sec:heavy}. \modkzz{Finally, we present in the appendix the calculations that allows us to deduce that  while the choice \(\bar b = 2\) in \eqref{eq:Pol_ODE1} is optimal if \(f\) is only assumed to be \(m\)-strongly convex, slightly better rates of convergence may be achieved for \(f\in \mathcal{F}_{m,L}\) by taking \(\bar b>2\)}.

\section{Preliminaries}\label{sec:prelim}
We will now briefly describe the framework introduced in \cite{FRMP} for the construction of Lyapunov functions of optimization methods and differential equations. The presentation here is adapted from the material in \cite{FRMP} to suit our specific needs.
{\color{black}
\begin{remark} The following material is limited to results needed to study strongly convex optimization. However the LMI approach  in \cite{FRMP} also works in convex optimization.
\end{remark}
}%end of color blue
\subsection{Optimization methods}
Optimization algorithms can often
be represented as linear dynamical systems interacting with one or
more static nonlinearities (see \cite{LRP16}). In this paper we will consider  first-order
algorithms that have the following state-space representation
\begin{subequations} \label{eq:control_disc}
\begin{align}
\xi_{k+1} &=A\xi_{k}+Bu_{k}, \\
u_{k} & = \nabla f(y_{k}), \\
y_{k} & =C\xi_{k}, \\
x_{k} &=E \xi_{k},
\end{align}
\end{subequations}
where $\xi_{k} \in \R^{n}$ is the state, $u_{k} \in \R^{d}$ is the input
$(d \leq n)$, $y_{k}  \in \R^{d}$ is the feedback output that is mapped  to
$u_{k}$ by the nonlinear map $\nabla f$. From the perspective of the optimization, $x_{k}$ is the
approximation to the mimimizer \(x^\star\).

As example, consider algorithms of the well-known form (\cite{LRP16,FRMP})
\begin{subequations} \label{eq:gen_eq}
\begin{align}
x_{k+1} &=x_{k}+\beta (x_{k}-x_{k-1})-\alpha \nabla f(y_{k}), \\
y_{k} &=x_{k}+\gamma(x_{k}-x_{k-1}),
\end{align}
\end{subequations}
where $\alpha>0,\beta,\gamma$ are scalar parameters that specify the algorithm within the family.
  For $\beta=\gamma=0$ we recover GD. For  $\beta=\gamma$, we have Nesterov's method; \eqref{eq:nest} corresponds to a particular choice of \(\alpha\) and \(\beta\). The Heavy Ball method has $\gamma=0$, \(\beta\neq 0\). By defining the state vector  $\modk{\xi_{k} =[x^{\T}_{k-1}, x^{\T}_{k}]^{T}} \in \R^{2d}$ we can represent \eqref{eq:gen_eq} in the  form \eqref{eq:control_disc} with the matrices $A,B,C, E$  given by
\[
A=\left[\begin{matrix}
0 & I_{d} \\
-\beta I_{d} & (\beta+1)I_{d}
\end{matrix}\right], \:
 B =\left[\begin{matrix} 0 \\-\alpha I_{d}\end{matrix}\right], \:
 C=\left[\begin{matrix} -\gamma I_{d} & (\gamma+1)I_{d} \end{matrix}\right]
,\:
E = \left[ \begin{matrix} 0 & I_d  \end{matrix}\right].
\]
Fixed points of \eqref{eq:control_disc} satisfy
\[
\xi^\star=A\xi^\star+Bu^\star, \quad y^\star=C\xi^\star, \quad u^\star=\nabla f(y^\star), \quad x^\star=E\xi^\star;
\]
in the optimization context \(u^\star  = 0\), and \(y^\star=x^\star\) is the minimizer sought.

To study the convergence rate of optimization algorithms, \cite{FRMP} considers  functions of the form
\begin{equation} \label{eq:liap_disc}
V_{k}(\xi)=\rho^{-2 k} \left(a_{0}(f(x)-f(x^\star))+(\xi-\xi^\star)^{\T}
P(\xi-\xi^\star) \right),
\end{equation}
where \(a_0 >0 \) and \(P\) is positive semi-definite (denoted by \(P \succeq 0\)).
If along the trajectories of \eqref{eq:control_disc}
\begin{equation} \label{eq:liap_decay}
V_{k+1}(\xi_{k+1}) \leq V_{k}(\xi_{k}),
\end{equation}
 we can conclude that $\rho^{-2k}a_{0}(f(x_{k})-f(x^\star)) \leq V_{k}(\xi_{k}) \leq V_{0}(\xi_{0})$ or
\[
 f(x_{k})- f(x^\star) \leq \rho^{2k} \frac{V_{0}(\xi_{0})}{a_{0}}.
\]
If  $\rho <1$, we have found a convergence rate for  $f(x_{k})$ towards the optimal value $f(x^\star)$. The following theorem defines an LMI  that, when $f \in \mathcal{F}_{m,L}$,  guarantees that the property \eqref{eq:liap_decay} holds and therefore \eqref{eq:liap_disc} provides a Lyapunov function for the system .

\begin{theorem} \label{theo:main_disc}
{\color{black}\emph{(Theorem 3.2 in \cite{FRMP}.)}
}%end of color blue
Suppose that, for \eqref{eq:control_disc}, there exist $a_{0} >0, P \succeq 0$, $\ell >0$, and $\rho \in [0,1)$
such that
\begin{equation} \label{eq:LMI_disc}
T= M^{(0)}
+a_{0}\rho^{2}M^{(1)}+a_{0}(1-\rho^{2})M^{(2)}+\ell M^{(3)} \preceq 0,
\end{equation}
where
\[
M^{(0)}=\left[ \begin{matrix}
A^{\T}PA-\rho^{2}P & A^\T PB \\
B^{\T}PA & B^{\T} P B
\end{matrix}\right],
\]
and
\[
M^{(1)}=N^{(1)}+N^{(2)}, \quad M^{(2)}=N^{(1)}+N^{(3)},\quad  M^{(3)}=N^{(4)},
\]
with
\begin{align*}
N^{(1)} &= \left[ \begin{matrix}
EA-C & EB \\
0 & I_{d}
\end{matrix}\right]^{\T}\left[ \begin{matrix}
\frac{L}{2}I_{d} & \frac{1}{2}I_{d} \\
\frac{1}{2}I_{d} & 0
\end{matrix}\right]\left[ \begin{matrix}
EA-C & EB \\
0 & I_{d}
\end{matrix}\right],
\\
N^{(2)} &= \left[ \begin{matrix}
C-E & 0 \\
0 & I_{d}
\end{matrix}\right]^{\T}\left[ \begin{matrix}
-\frac{m}{2}I_{d} & \frac{1}{2}I_{d} \\
\frac{1}{2}I_{d} & 0
\end{matrix}\right]\left[ \begin{matrix}
C-E & 0 \\
0 & I_{d}
\end{matrix}\right],
\\
N^{(3)} &=\left[ \begin{matrix}
C^{\T} & 0 \\
0 & I_{d}
\end{matrix}\right]\left[ \begin{matrix}
-\frac{m}{2}I_{d} & \frac{1}{2}I_{d} \\
\frac{1}{2}I_{d} & 0
\end{matrix}\right]\left[ \begin{matrix}
C & 0 \\
0 & I_{d}
\end{matrix}\right],
\\
N^{(4)} &=  \left[ \begin{matrix}
C^{\T} & 0 \\
0 & I_{d}
\end{matrix}\right]\left[ \begin{matrix}
-\frac{m L}{m+L}I_{d} & \frac{1}{2}I_{d} \\
\frac{1}{2}I_{d} &  -\frac{1}{m+L}I_{d}
\end{matrix}\right]\left[ \begin{matrix}
C & 0 \\
0 & I_{d}
\end{matrix}\right].
\end{align*}
Then, for $f \in \mathcal{F}_{m,L}$, the sequence $\{x_{k}\}$ satisfies
\[
f(x_{k})-f(x^\star) \leq \frac{a_{0}(f(x_{0})-f(x^\star))+(\xi_{0}-\xi^\star)^{\T}P(\xi_{0}-\xi^\star)}{a_0}\rho^{2 k}.
\]
\end{theorem}

\subsection{Continuous-time systems}
We  also consider  continuous-time dynamical systems in state space form
(throughout the paper we  often use a bar over symbols related to ODEs)
\begin{equation} \label{eq:con_system}
\dot{\xi}(t)=\bar{A}\xi(t)+\bar{B}u(t), \quad y(t)=\bar{C}\xi(t), \quad u(t)=\nabla f(y(t)) \quad \text{for all} \ t\geq0
\end{equation}
where  $\xi(t) \in \R^{n}$ is the state, $y(t) \in \R^{d} (d \leq n)$ the output, and $u(t)=\nabla f(y(t))$  the continuous feedback input.  Fixed points of \eqref{eq:con_system} satisfy
\[
0=\bar{A}\xi^\star, \quad y^\star=\bar{C}\xi^\star, \quad u^\star=\nabla f(y^\star);
\]
in our context \(u^\star = 0\) and \(y^\star =x^\star\).
We can replicate the convergence analysis of the discrete case using now  functions of the form
\begin{equation} \label{eq:liap_con}
\bar V(\xi(t)) =e^{\lambda t} \left(f(y(t))-f(y^\star)+(\xi(t)-\xi^\star)^{\T}\bar{P}(\xi(t)-\xi^\star) \right),
\end{equation}
where $\lambda >0$. If $\bar{P} \succeq 0$ and, along solutions, $(d/dt)\bar{V}(\xi(t)) \leq 0$, then we have $\bar V(\xi(t)) \leq \bar V(\xi(0))$ which in turns implies
\[
 f(y(t))-f(y^\star) \leq e^{-\lambda t} \bar V(\xi(0)).
\]

The following theorem similarly to the discrete time case, formulates an LMI that guarantees the existence of such a Lyapunov function.
\begin{theorem} \label{theo:main_con} Suppose that, for \eqref{eq:con_system}, there exist $\lambda >0$, $\bar{P} \succeq 0$, and $\sigma \geq 0$ that satisfy
\begin{equation} \label{eq:LMI_con}
\bar{T}= \bar M^{(0)}+\bar M^{(1)}+\lambda \bar M^{(2)}+\sigma \bar M^{(3)} \preceq 0
\end{equation}
where
\begin{align*}
\bar M^{(0)} &=
\left[ \begin{matrix}
\bar{P}\bar{A}+\bar{A}^{\T}\bar{P}+\lambda \bar{P} & \bar{P}\bar{B} \\
\bar{B} ^{\T}\bar{P}  & 0
\end{matrix}\right], \\
\bar M^{(1)} &= \frac{1}{2}\left[ \begin{matrix}
0 & (\bar{C} \bar{A} )^{\T} \\
\bar{C} \bar{A} & \bar{C} \bar{B} +\bar{B} ^{\T}\bar{C} ^{\T}
\end{matrix}\right], \\
\bar M^{(2)}&=\left[ \begin{matrix}
\bar{C}^{\T} & 0 \\
0 & I_{d}
\end{matrix}\right]\left[ \begin{matrix}
-\frac{m}{2}I_{d} & \frac{1}{2}I_{d} \\
\frac{1}{2}I_{d} & 0
\end{matrix}\right]\left[ \begin{matrix}
\bar{C} & 0 \\
0 & I_{d}
\end{matrix}\right], \\
\bar M^{(3)} &=\left[ \begin{matrix}
\bar{C}^{\T} & 0 \\
0 & I_{d}
\end{matrix}\right]\left[ \begin{matrix}
-\frac{m L}{m+L}I_{d} & \frac{1}{2}I_{d}, \\
\frac{1}{2}I_{d} & -\frac{1}{m+L}I_{d}
\end{matrix}\right]\left[ \begin{matrix}
\bar{C} & 0 \\
0 & I_{d}
\end{matrix}\right].
\end{align*}
Then the following inequality holds for $f \in \mathcal{F}_{m,L}$,  $t \geq 0$,
\[
f(y(t))-f(y^\star) \leq e^{-\lambda t} \left(f(y(0))-f(y^\star)+(\xi(0)-\xi^\star)^{\T}\bar{P}(\xi(0)-\xi^\star) \right).
\]
\end{theorem}
\section{A Lyapunov function for Nesterov's optimization algorithm} \label{sec:nest_anal}
We  study the  optimization  method (cf.\ \eqref{eq:gen_eq})
\begin{subequations}\label{eq:nest1}
\begin{align}
x_{k+1} &=x_k+\beta(x_k-x_{k-1})-\alpha\nabla f(y_k),\\
y_k  &= x_k+\beta  (x_k-x_{k-1}),
\end{align}
\end{subequations}
\(k=0,1,\dots\), with parameters \(\alpha>0\) and \(\beta\). As noted before, the
choice \(\beta=0\) gives GD and \(\beta\neq 0\)
corresponds to Nesterov's accelerated algorithm.

\subsection{The construction}

After introducing
\[
\delta = \sqrt{m\alpha},
\]
and the divided difference, \(k=0, 1,\dots\),
\begin{equation}\label{eq:dd}
d_k = \frac{1}{\delta}(x_k-x_{k-1}),
\end{equation}
%\end{equation}
the recursion \eqref{eq:nest1} may be rewritten (\(k=0,1,\dots\))
\begin{subequations}\label{eq:nest3}
\begin{align}
d_{k+1} &=\beta d_k -\frac{\alpha}{\delta} \nabla f(y_k),\\
x_{k+1}  &= x_k+\delta\beta d_k -\alpha \nabla f(y_k) ,\\
y_k &= x_k+\delta\beta d_k.
\end{align}
\end{subequations}
\begin{remark}\label{rem:dimensional} For future reference, it is useful to observe that, from a dimensional analysis point of view,
\(m\), \(L\) and \(1/\alpha\) have the dimensions of the quotient \(f/\|x\|^2\). Therefore \(\delta\) is a \emph{non-dimensional} version of  \(\sqrt{\alpha}\). The parameter \(\beta\) is non-dimensional. The divided difference \eqref{eq:dd} shares  the dimensions of \(x\).
\end{remark}

Equation \eqref{eq:nest3} can now be written in the form \eqref{eq:control_disc} with $\xi_k= [d_k^{\T},x_k^{\T}]^{\T} \in\R^{2d}$ and
\begin{equation}\label{eq:A}
A = \left[\begin{matrix}\beta I_d & 0\\ \delta\beta I_d& I_d\end{matrix}\right],\quad
B = \left[\begin{matrix}-(\alpha/\delta) I_d\\ -\alpha I_d\end{matrix}\right],\quad
C = \left[\begin{matrix} \delta\beta I_d & I_d\end{matrix}\right],\quad
E = \left[\begin{matrix} 0 & I_d\end{matrix}\right].
\end{equation}

In the preceding section, as in \cite{FRMP}, the state \(\xi_k\) was taken to be \([x_{k-1}^{\T},x_k^{\T}]^{\T}\) rather than \([d_k^{\T},x_k^{\T}]^{\T}\). While both choices are of course mathematically equivalent, the new \(\xi_k\) is more convenient for our purposes. In addition, when looking numerically for Lyapunov functions by solving LMIs, it leads to problems that are better conditioned for large condition numbers \(\kappa\).

\begin{remark}\label{rem:gradient}For \(\beta=0\) (gradient descent), the first equation in \eqref{eq:nest3} is a reformulation of the second: it would be more natural to use the simpler state \(\xi_k=x_k\).
\end{remark}

According to Theorem~\ref{theo:main_disc}, in order  to find a
Lyapunov function of the form \eqref{eq:liap_disc}, it is sufficient to find a matrix $P \succeq 0$ and numbers  $a_{0}>0$,  $0<\rho<1$,
$\ell \geq 0$, such that the matrix $T$ in \eqref{eq:LMI_disc} is negative semi-definite. At the outset, we choose
\(\ell = 0\) in order to simplify the subsequent analysis. As we will discuss in the Appendix, this simplification does not have  a significant impact on the value of the convergence rate \(\rho\) that results from the analysis. With \(\ell = 0\),
\eqref{eq:LMI_disc} is homogeneous in \(P\) and \(a_0\) and we may divide accross by \(a_0\). In other words, without loss of generality, we may take \(a_0=1\). Then \(T\) is a function of \(P\) and \(\rho\) (and the method parameters \(\beta\) and \(\delta\)).

The matrix \(A\) in \eqref{eq:A} is a Kronecker product of a \(2\times 2\) matrix and \(I_d\),
\[
A = \left[\begin{matrix}\beta  & 0\\ \delta\beta & 1\end{matrix}\right]\otimes I_d;
\]
the factor \(I_d\) originates from the dimensionality of the decision variable \(x\) and the \(2\times 2\) factor  is independent of \(d\) and arises from the optimization algorithm. The matrices \(B\), \(C\) and \(E\) have a similar Kronecker product structure.  It is then natural to consider symmetric matrices \(P\) of the form
\begin{equation}\label{eq:P}
P = \widehat P \otimes I_d,\qquad
\widehat P = \left[\begin{matrix}p_{11}  & p_{12}\\ p_{12} & p_{22}\end{matrix}\right],
\end{equation}
and then \(T\) will also have a Kronecker product structure
\begin{equation}\label{eq:T}
T = \widehat T\otimes I_d,\qquad
\widehat T = \left[\begin{matrix}t_{11}  & t_{12}&t_{13}\\ t_{12} & t_{22}&t_{23}\\t_{13} & t_{23}&t_{33}\end{matrix}\right],
\end{equation}
where the \(t_{ij}\) are explicitly given by the following complicated expressions obtained from \eqref{eq:A} and the recipes for \(M^{(0)}\), \(M^{(1)}\) and \(M^{(2)}\) in Theorem~\ref{theo:main_disc}:
\begin{subequations}
\label{eq:nuevoT}
\begin{align}
  t_{11} &= \beta^2p_{11}+2\delta \beta^2p_{12}+\delta^2 \beta^2p_{22}-\rho^2 p_{11}-\delta^2\beta^2m/2, \\
  t_{12} &= \beta p_{12}+\delta \beta p_{22}-\rho^2 p_{12} -\delta\beta m/2+\rho^2 \delta \beta m/2,\\
  t_{13} &=-\delta^{-1}\alpha\beta  p_{11}-2\alpha\beta p_{12}-\delta \alpha \beta p_{22} +\delta\beta/2, \\
  t_{22} &= p_{22}-\rho^2 p_{22}-m/2+\rho^2 m/2,  \\
  t_{23} &= -\delta^{-1}\alpha p_{12}-\alpha p_{22}+1/2-\rho^2/2, \\
  t_{33}&=  \delta^{-2}\alpha^2 p_{11}+2\delta^{-1}\alpha^2p_{12}+\alpha^2p_{22}+\alpha^2L/2-\alpha.
\end{align}
\end{subequations}

Our task is to  find \(\rho\in[ 0,1)\), \(p_{11}\), \(p_{12}\), and \(p_{22}\) that lead to \(\widehat T \preceq 0\)
and \(\widehat P \succeq 0\) (which imply \(T \preceq 0\)
and \(P \succeq 0\) ).
The algebra becomes simpler if we represent  \(\beta\) and \(\rho^2\) as:
\begin{equation}\label{eq:betarhosq}
\beta = 1-b\delta,\qquad
\rho^2 = 1-r\delta.
\end{equation}
Note that we are interested in \(r\in (0,1/\delta]\) so as to get \(\rho^2\in [0,1)\).
We proceed in steps as follows.

{\em First step.} Impose the condition \(t_{23}=0\). This leads to
\begin{equation}\label{eq:p12}
p_{12} = \frac{m}{2} r-\delta p_{22}.
\end{equation}

{\em Second step.} Impose the condition \(t_{13}=0\).  This results in
\[
p_{11} = \frac{m}{2}-2\delta p_{12} -\delta^2 p_{22},
\]
which in tandem with \eqref{eq:p12} yields
\begin{equation}\label{eq:p11}
p_{11} = \frac{m}{2}-mr\delta+\delta^2 p_{22}.
\end{equation}

{\em Third step.} Impose the condition \({\rm det}(\widehat P)=p_{11}p_{22}-p_{12}^2=0\). Using \eqref{eq:p12} and \eqref{eq:p11}, we have a linear equation for \(p_{22}\) with solution
\[
p_{22} = \frac{m}{2}r^2.
\]
We now take this value to \eqref{eq:p12} and \eqref{eq:p11} and get
\begin{equation}\label{eq:Pbis}
\widehat P = \left[\begin{matrix}p_{11} &p_{12}\\p_{12} &p_{22} \end{matrix} \right]
=\frac{m}{2} \left[\begin{matrix}(1-r\delta)^2 &r(1-r\delta)\\r(1-r\delta) &r^2 \end{matrix}\right],
\end{equation}
a matrix that is positive semi-definite (but not positive definite).

{\em Fourth step.} Impose \(t_{33}\leq 0\). After using  \eqref{eq:Pbis} in the expression for \(t_{33}\) in \eqref{eq:nuevoT}, this condition  is  seen to be equivalent to \(\alpha^2L-\alpha\leq 0\) or
\[\alpha \leq\frac{1}{L}
\]
(for \(\alpha=1/L\), \(t_{33}\) actually vanishes). In what follows we assume that this bound on \(\alpha\) holds; note that then \(\delta = \sqrt{m\alpha}\leq \sqrt{m/L}< 1\).

{\em Fifth step.} We impose \(t_{22}\leq 0\). This may be written as \((p_{22}-m/2)r\delta\leq 0\), which  leads to \(p_{22}\leq m/2\). From \eqref{eq:Pbis}
\[r\leq 1,
\]
which sets a lower limit \(\rho^2\leq 1-\delta\) for the rate of convergence. For \(r^2<1\), \(t_{22}<0\).

{\em Sixth step.} Impose \(t_{11}t_{22}-t_{12}^2=0\). From \eqref{eq:Pbis} and \eqref{eq:nuevoT}, some algebra yields
\[
t_{11}t_{22}-t_{12}^2 = -\frac{m^3}{4} r (1-r\delta)\: \Xi
\]
with
\begin{equation}\label{eq:Xi}
\Xi = \Xi_\delta(r,b)= (r+\delta)(1-\delta^2) b^2 -2 (1+r^2)(1-\delta^2) b +(r^3-3r^2\delta+3r-\delta).
\end{equation}
Since \(\delta <1\) and, after step five, \(r\in(0,1]\),
 we must have \(\Xi=0\). For fixed \(\delta \in(0,1)\), the condition \(\Xi_\delta=0\) establishes a relation between the values of \(r\) and \(b\) or, in other words, the rate of convergence \(\rho^2\) and the parameter \(\beta\) in \eqref{eq:nest1}. In order to study this relation, we now make a digression and describe, for fixed \(\delta\in(0,1)\),  the algebraic curve of equation \(\Xi_\delta(r,b)=0\) in the real plane \((r,b)\); in this description we allow arbitrary real values of \(r\) and \(b\) (even though in our problem \(r\in(0,1]\)).

The formula for the roots of a quadratic equation yields
\begin{equation}\label{eq:roots}b_{\pm} = \frac{(1+r^2)(1-\delta^2) \pm (1-r\delta) \sqrt{(1-r^2) (1-\delta^2)}}
{(r+\delta)(1-\delta^2)}.
\end{equation}
For  \(r^2\neq 1\) and \(r\neq -\delta\) there are two distinct real roots \(b_+\) and \(b_-\). For \(r=\pm 1\) there is a double root \(b= 2/(r+\delta)\). As \(r\downarrow -\delta\), we have \(b_+ \uparrow +\infty\) and \(b_- \downarrow -2\delta/(1-\delta^2)\). By using \eqref{eq:roots} it is not difficult to prove that \(\Xi_\delta(r,b)=0\) defines \(r\) as a single-valued function of the variable \(b\in\R\). (We could find an explicit expression for \(r\) in terms of \(b\) by means of the formula for the roots of a cubic equation, but this is not necessary for our purposes.) Figure~\ref{figure} provides a plot of the curve \(\Xi_\delta(r,b)=0\) when \(\delta =1/2\).

We now return to the construction of \(T\). Recall that for our purposes, we need \(r>0\) (so as to have \(\rho<1\)); this requirement holds for \(b\in(b_{\rm min}, b_{\rm max})\), where
\[
b_{\rm min} =\frac{1-\delta^2-\sqrt{1-\delta^2}}{\delta(1-\delta^2)}<0,\qquad
b_{\rm max} =\frac{1-\delta^2+\sqrt{1-\delta^2}}{\delta(1-\delta^2)}>0,
\]
are the intersections of the curve \(\Xi_\delta = 0\) with the vertical axis.
As \(\delta \downarrow 0\),
 \begin{equation}\label{eq:limits}
 b_{\rm min}\uparrow 0,\qquad b_{\rm max}\uparrow +\infty.
 \end{equation} The limits on \(b\) just found are equivalent to
\begin{equation}
-\sqrt{1-\delta^2} < \beta < +\sqrt{1-\delta^2}.
\end{equation}\label{eq:minmax}
For the maximum value \(r=1\) found in step five above, the formula
\eqref{eq:roots} gives the double root
 \(b= 2/(1+\delta)\) or \(\beta = (1-\delta)/(1+\delta)\). Values \(r\in(0,1)\) correspond to two different choices of
 \(b\in(b_{\rm min}, b_{\rm max})\).

\begin{figure}[t]
%\vskip -4cm
\begin{center}
\includegraphics[width=0.9\hsize]{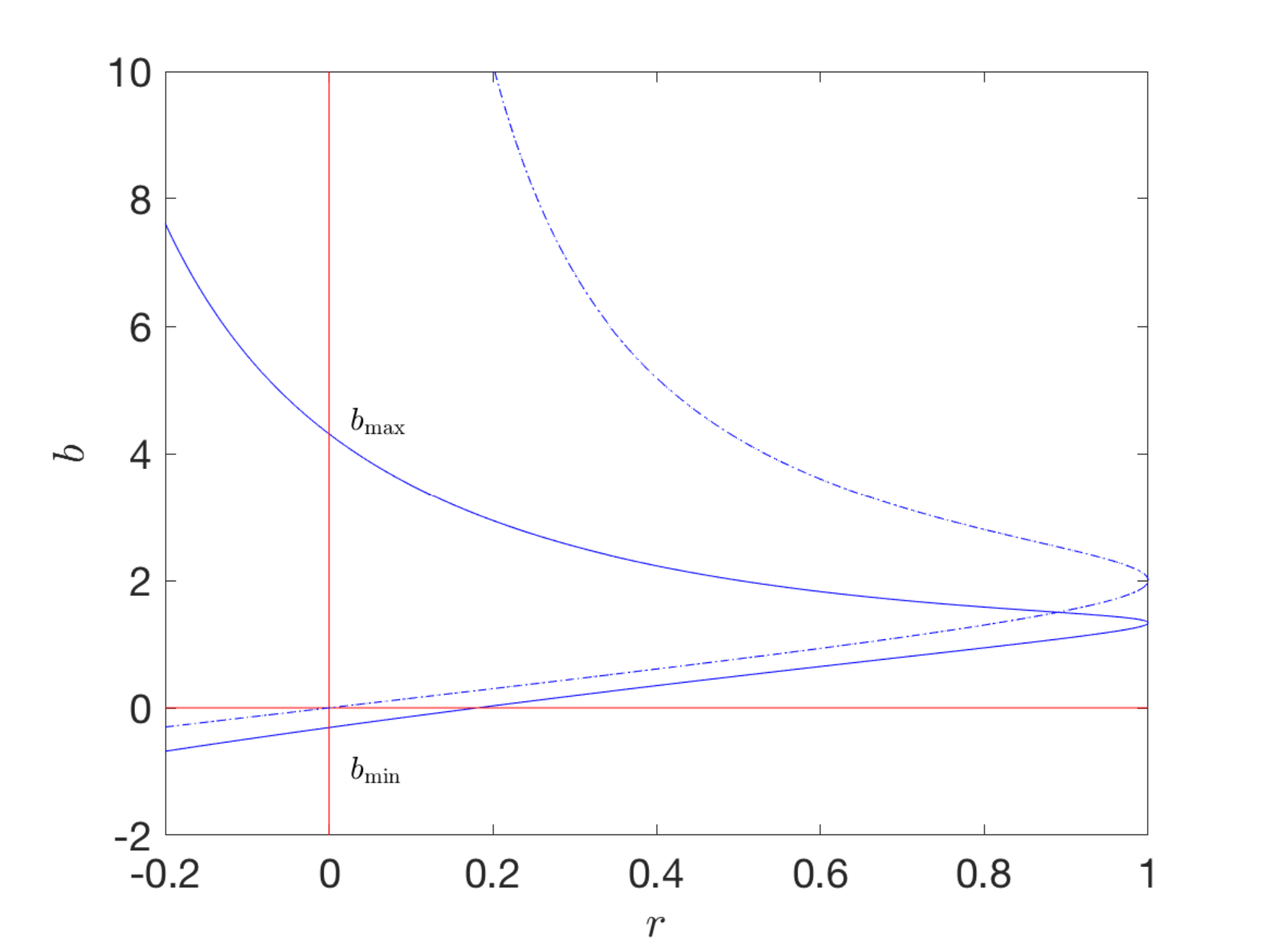}
\end{center}
\label{figure}
%\vskip -5cm
\caption{The solid curve corresponds to the equation \(\Xi_\delta(r,b)=0\)  when \(\delta =1/2\). It has a vertical asymptote at \(r=-\delta\) (not shown). To each real \(b\) there corresponds a single value of \(r\). For \(b\in(b_{\rm min},b_{\rm max})\), we have \(0<r \leq 1\), that corresponds to
\(1 >\rho^2 \geq 1-\delta\). The best rate \(\rho^2=1-\delta\) is achieved for \(b= 2\delta/(1+\delta)\), i.e.\ \(\beta = (1-\delta)/(1+\delta)\). The discontinuous curve corresponds to the equation \(\Xi_\delta(r,b)=0\) in the limit \(\delta \rightarrow 0\); again to each real \(b\) there corresponds a single value of \(r\). This curve is symmetric with respect to the origin (changing \(b\) into \(-b\) changes \(r\) into \(-r\)) and  has a vertical asymptote at \(r=0\). Positive values of \(b\) correspond to positive values of \(r\). The maximum value \(r=1\) is achieved when \(b=2\).}
\end{figure}

We are now ready to present the following result.

\begin{theorem}
\label{th:1}Consider the minimization algorithm \eqref{eq:nest1} (or \eqref{eq:nest3}) with parameters subject to
\[
\alpha \leq 1/L, \qquad -\sqrt{1-{m\alpha}} \leq \beta \leq \sqrt{1-{m\alpha}}.
\]
Set \(\delta = \sqrt{m\alpha}\) and  let \(r>0\) be the value determined
by \(\Xi_\delta(r,b) = 0\) (see \eqref{eq:Xi}), set \(\rho^2 = 1-r\delta<1\) and define the positive semi-definite matrix \(P\) by
\eqref{eq:P} and \eqref{eq:Pbis}. Then the matrix \(T\) in \eqref{eq:T}--\eqref{eq:nuevoT} is negative semi-definite.

As a result, for any \(x_{-1}\), \(x_0\),
the sequence
\begin{equation}\label{eq:theo}
\rho^{-2k}\Big( f(x_k)-f(x_\star)+ [d_k^{\T},x_k^{\T}-x_\star^{\T}]\,P\,[d_k^{\T},x_k^{\T}-x_\star^{\T}]^{\T}\Big)
\end{equation}
decreases monotonically, which, in particular, implies \[f(x_k)-f(x_\star)\leq C\rho^{2k}\]
with
\[
C = f(x_0)-f(x^\star) + \frac{m}{2}\left\|\frac{1-r\delta}{\delta} (x_0-x_{-1})+r(x_0-x^\star) \right\|^2.
\]  \end{theorem}
\begin{proof} Using Theorem \ref{theo:main_disc}, we only have to prove that \(\widehat T\preceq 0\). The  second, first and fourth steps of our construction respectively ensure that \(t_{13}=t_{23}=0\) and \(t_{33}\leq 0\) and therefore we are left with the task of checking that the \(2\times 2\) matrix
\(\widehat T^{12}\) obtained by suppressing the last row and last column of \(\widehat T\) is \(\preceq 0\). If \(r<1\), we know from step five that \(t_{22} <0\) and from step six that the determinant of \(\widehat T^{12}\) vanishes and therefore \(\widehat T^{12}\preceq 0\). For \(r=1\), \(t_{22} =0\), but again \(\widehat T^{12}\preceq 0\), because in this case \(t_{11} = - (m/2) \delta (1-\delta)^3/(1+\delta)<0\).
\end{proof}

For fixed \(\alpha\leq 1/L\), as noted above,  \(\rho^2\) is minimized by the choice \[\beta= (1-\sqrt{m\alpha})/(1+\sqrt{m\alpha});\] then
\[
\rho^2 = {1-\sqrt{m\alpha}}.
\]
 When \(\alpha\) is allowed to vary in the interval \((0,1/L]\), increasing  \(\alpha\) results in an improvement of \(\rho^2\), so that the best rate \(\rho^2 = 1-\sqrt{m/L}=1-\sqrt{1/\kappa}\) is obtained by setting \(\alpha=1/L\) and then \eqref{eq:nest1} coincides with \eqref{eq:nest}. The parameter values \(\alpha=1/L\), \(\beta = (1-\sqrt{1/\kappa})/(1+\sqrt{1/\kappa})\) in \eqref{eq:nest} are of course the \lq\lq standard\rq\rq\  choice for Nesterov's algorithm (see e.g.\  \cite[Proposition 12]{LRP16}). For this choice of parameters and \(x_{-1} = x_0\), the bound in Theorem~\ref{th:1} exactly coincides
 (including the value of \(C\)) with that in \eqref{eq:boundfrombook}, which is
 derived in \cite[Theorem 2.2.3]{N14} without using Lyapunov functions. Numerical experiments in \cite{LRP16}
show that for \(\kappa^{-1} = m/L\) small the rate of convergence \(\rho^2 = 1-\sqrt{1/\kappa}\) is essentially the best that the algorithm achieves.

The theorem may also be applied to the GD algorithm with \(\beta=0\) and
\(b=1/\delta\), even though (see Remark~\ref{rem:gradient}) in this case the preceding treatment is unnatural. One finds \(r=\delta\), so that the decay per step in \( f(x_k)-f(x_\star)\) provided by Theorem~\ref{th:1} is
\(
\rho^2 = 1-\delta^2 = 1- m\alpha
\), for \(\alpha \leq 1/L\).
\modk{When \(\alpha=2/(m+L)\), the decay per step guaranteed by Theorem~\ref{th:1} is \(\rho^2 = \frac{1-1/\kappa}{1+1/\kappa}\); this is worse than the bound in \eqref{eq:GDbound} valid for the same value of \(\alpha\)}.

\begin{remark}\label{rem:dim2} The decay rate \(\rho^2\) provided by the theorem is a non-dimensional quantity that only depends on the non-dimensional variables \(b\) and \(\delta\). The bound \(\alpha\leq 1/L\) may be rewritten in the non-dimensional form as \(\delta^2\leq m/L=1/\kappa\). These facts guarantee that the theorem is equivariant with respect to changes in scale of \(f\) and \(x\). The Lyapunov function in \eqref{eq:theo} has the dimensions of \(f\) because, according to \eqref{eq:Pbis}, \(P\) has the dimensions of \(m\), i.e.\ those of \(f/\|x\|^2\).
\end{remark}

{\color{black}\begin{remark}\label{rem:comparisonLyapunov} For the particular choice of \(\alpha\) and \(\beta\) leading to \eqref{eq:nest}, the Lyapunov function in the theorem above was derived in
\cite{LO20}  by means of an alternative technique (see Remark~\ref{rem:modified}).
In \cite{SDJ18} a Lyapunov function that contains the gradient \(\nabla f(x)\) is constructed analytically for the situation where the learning rate \(\alpha\) in \eqref{eq:nest1} is a free parameter and the momentum parameter is fixed as \(\beta = (1-\sqrt{m\alpha})/(1+\sqrt{m\alpha})\) (i.e.\ at the value that according to the analysis above
optimizes \(\rho^2\)). The analysis in \cite{SDJ18} requires (see Lemma 3.4 in that reference) \(\alpha \leq 1/(4L)\), while here \(\alpha\leq 1/L\).
In addition for \(\alpha= 1/(4L)\), \cite[Theorem 3]{SDJ18} proves a rate \(1/(1+(1/12)\sqrt{m/L})\)
which, while establishing acceleration, compares unfavourably with the value \(1-(1/2)\sqrt{m/L}\) provided by Theorem~\ref{th:1}. %\modkzz{As we will show in the next section this relates directly to the choice of $\widehat P$, and the fact that our choice is optimal in terms of achieving the best possible rate within this class of Lyapunov functions.}
\end{remark}
}

\subsection{Optimality}

The path leading to Theorem~\ref{th:1} has a degree of arbitrariness and it may be asked whether,
by following an alternative construction,  it is possible
to determine the parameters \(\rho\), \(p_{11}\), \(p_{12}\), \(p_{22}\) and  in such a way that \(\widehat T\preceq 0\), \(\widehat P\succeq 0\)
and the value of \(\rho\) is larger than the value provided in Theorem~\ref{th:1}. We conclude this section by presenting a result in this direction. {\color{black}We fix the parameters in the algorithm at the standard choices
 i.e.\ \(\alpha=1/L\), \(\beta=(1-\delta)/(1+\delta)\), \(\delta=\sqrt{m/L}\), and
denote by \(\rho^\star = \sqrt{1-\delta}\), \(p^\star_{11} = (m/2) (1-\delta)^2\), \(p^\star_{12} = (m/2)(1-\delta)\), \(p^\star_{22} = m/2\) the values yielded by Theorem~\ref{th:1}.
In the space of the decision variables \(\rho\), \(p_{11}\), \(p_{22}\), \(p_{33}\) we pose  the convex optimization problem of minimizing \(\rho\) subject to the constraints \(\widehat T\preceq 0\), \(\widehat P\succeq 0\).
We then have the following result that shows that the rate provided in Theorem~\ref{th:1} cannot be improved with an alternative choice of \(\widehat P\).

\begin{theorem} With the notation just described, the unique solution of the minimization problem is \((\rho^\star, p_{11}^\star,   p_{12}^\star, p_{22}^\star)\).
\end{theorem}

\begin{proof}
We use the notation \(\sigma=\rho^2\),  \(\sigma^\star = (\rho^\star)^2\) and write  \(\sigma = \sigma^\star+\wsigma\), \(p_{11} = p_{11}^\star+\wpoo\), \(p_{12} = p_{12}^\star+\wpot\), \(p_{22} = p_{22}^\star+\wptt\). Since the minimization problem is convex, it is sufficient to show that \(\rho^\star\), \(p_{11}^\star\), \(p_{12}^\star\), \(p_{22}^\star\)  provide a local minimum, i.e.\ that if the increments \(\wsigma\leq 0\),
\(\wpoo\), \(\wpot\), \(\wptt\) are of sufficiently small magnitude and \((\sigma, p_{11}, p_{12}, p_{22})\) is feasible, then
\(\sigma = \sigma^\star\), \(p_{11} = p_{11}^\star\), \(p_{12} = p_{12}^\star\), \(p_{22} = p_{22}^\star\).

We study three requirements that feasibility imposes on  \(\wsigma\), \(\wpoo\), \(\wpot\), \(\wptt\).

(1) First, the constraint  \(\widehat P\succeq 0\) implies that \(p_{11}p_{22}-p_{12}^2\geq 0\) or
\[
p_{22}^\star \wpoo -2 p_{12}^\star \wpot +p_{11}^\star \wptt + \wpoo\wptt- (\wpot)^2\geq 0.
\]
Because we are carrying a local study, we replace the constraint by its linearization
\[
p_{22}^\star \wpoo -2 p_{12}^\star \wpot +p_{11}^\star \wptt \geq 0.
\]
or, after using the known values of the symbols with a star,
\begin{equation}\label{eq:consp}
 \wpoo -2 (1-\delta) \wpot +(1-\delta)^2 \wptt \geq 0.
\end{equation}

(2) Then, the constraint \(\widehat T\preceq 0\) implies \(t_{22}t_{33}- t_{23}^2\geq 0\) or, using \eqref{eq:nuevoT},
\begin{eqnarray*}
&&-\Big(\frac{1}{2} \wsigma+ \frac{\delta}{m} \wpot+\frac{\delta^2}{m} \wptt\Big)^2+  \frac{\delta^3}{m^2} \wptt \big(\wpoo+2\delta\wpot+\delta^2\wptt\big)\\&&\qquad\qquad\qquad\qquad\qquad\qquad\qquad\qquad\qquad
-\frac{\delta^2}{m^2} \wsigma\wptt \big(\wpoo+2\delta\wpot+\delta^2\wptt\big) \geq 0.
\end{eqnarray*}
This time the leading terms in the right hand-side are quadratic in the increments and we discard the cubic terms to get:
\begin{equation}\label{eq:cons23}
-\Big(\frac{m}{2} \wsigma+ \delta \wpot+\delta^2 \wptt\Big)^2+  \delta^3 \wptt \big(\wpoo+2\delta\wpot+\delta^2\wptt\big)\geq 0.
\end{equation}
By completing the square in the quadratic form, this may be equivalently rewritten as
\begin{equation}\label{eq:completesquare}
\Big(\frac{m}{2} \wsigma+ \delta \wpot+\delta^2 \wptt\Big)^2
+\delta\Big(\frac{1}{2} \wpoo+\delta \wpot\Big)^2
\leq \delta \Big(\frac{1}{2} \wpoo+\delta \wpot+\delta^2 \wptt\Big)^2.
\end{equation}

(3) Finally \(\widehat T \preceq 0\) requires \(t_{22}\leq 0\) or
\(
\wptt(\delta-\wsigma) \leq 0
\);
 discarding the quadratic term, we get
\begin{equation}\label{eq:consp22}
\wptt  \leq 0.
\end{equation}

The proof concludes by applying the lemma below.
\end{proof}
%%%%%%%%%%%%%%%%%%%%%%%%%%%%%%%%%%%%%%%%%%%%%%
\begin{lemma}If  the increments \(\wsigma\leq 0\), \(\wpoo\), \(\wpot\), \(\wptt\) satisfy the constraints
\eqref{eq:consp}--\eqref{eq:consp22}, then \(\wsigma= 0\), \(\wpoo=0\), \(\wpot=0\), \(\wptt=0\).
\end{lemma}
\begin{proof}
The relation \eqref{eq:completesquare} obviously implies
\[
\Big(\frac{1}{2} \wpoo+\delta \wpot\Big)^2
\leq  \Big(\frac{1}{2} \wpoo+\delta \wpot+\delta^2 \wptt\Big)^2
\]
and therefore, in view of \eqref{eq:consp22},
\begin{equation}\label{eq:aux1}
\frac{1}{2} \wpoo+\delta \wpot\leq 0.
\end{equation}
We combine this inequality with \eqref{eq:consp} to get
\[
0\leq -2  \wpot+(1-\delta)^2 \wptt
\]
so that
\begin{equation}\label{eq:cons12}
\wpot \leq 0.
\end{equation}

Since the three quantities being added in the first bracket in \eqref{eq:completesquare} are now known to be \(\leq 0\), it is enough to consider hereafter the worst case \(\wsigma = 0\).
\[
\Big( \delta \wpot+\delta^2 \wptt\Big)^2
\leq \delta \Big(\frac{1}{2} \wpoo+\delta \wpot+\delta^2 \wptt\Big)^2.
\]
Since \(\delta \wpot+\delta^2 \wptt\leq 0\), we must have
\begin{equation}\label{eq:cons11}
\wpoo \leq 0.
\end{equation}

From \eqref{eq:consp}
\[
\wpoo+2\delta \wpot+\delta^2 \wptt
 \geq 2  \wpot +(-1+2\delta) \wptt,
\]
which implies (see \eqref{eq:consp22}, \eqref{eq:cons12}, \eqref{eq:cons11})
\[
\wptt (\wpoo+2\delta \wpot+\delta^2 \wptt) \leq 2  \wpot\wptt +(-1+2\delta) \wptt^2.
\]
By  combining this inequality and \eqref{eq:cons23} (with \(\wsigma = 0\)), we obtain a relation
\[\delta^2 \wpot^2+\delta^3 (1-\delta) \wptt^2 \leq 0,
\]
that shows that \(\wpot=0\). Then comparing \eqref{eq:consp}, \eqref{eq:consp22} and \eqref{eq:cons11}, we conclude that \(\wpoo=\wptt=0\), \modkzz{which in turn concludes  the proof.}
\end{proof}
%end of blue color
}

\section{The differential equation}\label{sec:ode}

Let us now set \(h=\sqrt{\alpha}\) (so that \(\delta = \sqrt{m} h\)) and assume that in  \eqref{eq:nest1}, the parameter \(\beta=\beta_h\) changes smoothly with \(h\) in such a way that, for some constant \(\bar b\in\R\), \(\beta_h = 1- \bar b
\sqrt{m} h+o(h)\) as \(h\downarrow 0\). Then,  \eqref{eq:nest1} may be written as
\[
\frac{1}{h^2} (x_{k+1}-2x_k+x_{k-1})+ \frac{1-\beta_h}{\sqrt{m}h} \sqrt{m}\frac{1}{h}(x_k-x_{k-1})+\nabla f(y_k)=0,
\]
which, if \(x_k\) is seen as an approximation to \(x(kh)\), provides a consistent discretization of the differential equation
\eqref{eq:Pol_ODE1}. An example is provided by the choice
\(\beta = (1-\delta)/(1+\delta) = (1-\sqrt{m} h)/(1+\sqrt{m} h)\), where \(\bar b = 2\) and \eqref{eq:Pol_ODE1} is the equation
\eqref{eq:Pol_ODE} used by Polyak.

\begin{remark} \label{re:oneleg} In general, this two-step discretization is,  not a linear multistep formula. Note:
\begin{itemize}
\item \(\nabla f\) is evaluated at \(y_k\), a linear combination of \(x_k\) and \(x_{k-1}\). In this regard, \eqref{eq:nest1} is similar to the \emph{one-leg} methods introduced by Dahlquist in his study of the long-time properties of multistep methods applied to nonlinear differential equations  (see e.g.\ \modk{\cite{GD76,B16,HaW96}})
%*** REFERENcE  TO Hairer y Wanner THAT WAS REFERENCE [4] IN THE OTHER PAPER WE WROTE  ).
\item The unconventional factor \((1-\beta_h)/(\sqrt{m}h)\) that converges to \(\bar b\) as \(h\downarrow 0\). From the point of view of discretization methods for ODEs having \(\bar b\) instead of this factor, or equivalently having
    \(\beta = 1-\bar b \sqrt{m} h\), would be more natural. But note that, when \(\beta = (1-\sqrt{m} h)/(1+\sqrt{m} h)\), the algorithm \eqref{eq:nest1} becomes GD for \(h=1/\sqrt{L}\) and \(\kappa=1\); the choice \(\beta = 1-\bar b \sqrt{m} h\) does not share this favourable property.
\end{itemize}
\end{remark}

\subsection{The construction}% color blue has ended

We  now define
\[
v = \frac{1}{\sqrt{m}} \dot x
\]
and rewrite \eqref{eq:Pol_ODE1} as a first-order system
\begin{subequations}\label{eq:ode1}
\begin{align}
\dot v &= -\bar b \sqrt{m} v -\frac{1}{\sqrt{m}} \nabla f(x),\\% a typo here has been corrected 21 dec 2020
\dot x &= \sqrt{m} v.
\end{align}
\end{subequations}

\begin{remark}In  a dimensional analysis as in Remarks~\ref{rem:dimensional} and \ref{rem:dim2}, \(h\) has the same units as \(t\). It is then a dimensional time-step, to be compablue with the non-dimensional  \(\delta\). The units of \(v\) are those of \(x\). Of course,
the divided difference \eqref{eq:dd} is a discrete version of \(v=\dot x/\sqrt{m}\).
\end{remark}
If we  set \(\xi = [v^{\T},x^{\T}]^{\T}\), then \eqref{eq:ode1} is of
the form \eqref{eq:con_system} with
\[
\bar A = \left[\begin{matrix}-\bar b\sqrt{m}I_d & 0_d\\ \sqrt{m} I_d& 0_d\end{matrix}\right],\quad
 \bar B = \left[\begin{matrix}-(1/\sqrt{m}) I_d\\ 0_d\end{matrix}\right],
\quad
 \bar C = \left[\begin{matrix} 0_d & I_d\end{matrix}\right],
\]

Now according to Theorem~\ref{theo:main_con},  in order  to find a Lyapunov function of the form \eqref{eq:liap_con} it is sufficient to find a matrix $\bar{P} \succeq 0$ and parameters $\lambda>0$, $\sigma \geq 0$ such that the matrix $\bar{T}$ in \eqref{eq:LMI_con} is negative semi-definite. Similarly to the discrete case, we will simplify the subsequent analysis by considering the case $\sigma=0$. (The case \(\sigma >0\) is studied in the Appendix.) The Lipschitz constant \(L\) only enters \(T\) in Theorem~\ref{theo:main_con} through \(\bar M^{(3)}\); under the assumption \(\sigma =0\), \(\bar T\) is independent of \(L\). This has an important implication: the analysis in this section applies to \(f\) strongly \(m\)-convex but \emph {not necessarily \(L\)-smooth}.

We look for \(\bar P\) of the form
\begin{equation}\label{eq:Pcont}
\bar P = \widehat {\bar P} \otimes I_d,\qquad
\widehat {\bar P} = \left[\begin{matrix}\bar p_{11}  & \bar p_{12}\\ \bar p_{12} & \bar p_{22}\end{matrix}\right],
\end{equation}
and then \(\bar  T\) is found to be
\begin{equation}\label{eq:Tbis}
\bar T = \widehat {\bar T}\otimes I_d,\qquad
\widehat {\bar T} = \left[\begin{matrix}\bar t_{11}  & \bar t_{12}&\bar t_{13}\\ \bar t_{12} & \bar t_{22}&\bar t_{23}\\\bar t_{13} & \bar t_{23}&\bar t_{33}\end{matrix}\right],
\end{equation}
where the \(\bar t_{ij}\) have the following expressions:
\begin{align*}
  \bar t_{11} &=  -2\bar b \bar p_{11}+2\sqrt{m} \bar p_{12}+\lambda \bar p_{11},\\
  \bar t_{12} &= -\bar b \sqrt{m} \bar p_{12}+\sqrt{m} \bar p_{22}+\lambda \bar p_{12}, \\
  \bar t_{13} &= -(1/\sqrt{m})\bar p_{11}+\sqrt{m}/2, \\
  \bar t_{22} &= \lambda  \bar p_{22}-(m/2)\lambda,\\
  \bar t_{23} &= -(1/\sqrt{m})\bar p_{12}+\lambda/2,\\
  \bar t_{33}&=  0.
\end{align*}

We now determine \(\lambda\) and \(\widehat{\bar P}\).
The algebra  is simplified if we set \(\lambda = \sqrt{m}\:{\bar r}\).

{\em First step.} Since \(\bar t_{33}=0\), the requirement \(\widehat {\bar T}\preceq 0\) implies \(\bar t_{13}=0\) and \(\bar t_{23}=0\) and accordingly
\begin{equation}\label{eq:p11p12}
\bar p_{11} = m/2,\qquad \bar p_{12} = (m/2){\bar r}.
\end{equation}

{\em Second step.}
We choose \(\bar p_{22}\) to ensure \({\rm det}(\widehat{\bar P}) = \bar p_{11}\bar p_{22}-\bar p_{12}^2=0\). This yields
\[
\bar p_{22} = (m/2) {\bar r}^2,
\]
and leads to
\begin{equation}\label{eq:Pcontbis}
\widehat {\bar P} = \frac{m}{2} \left[\begin{matrix} 1&{\bar r}\\ {\bar r} & {\bar r}^2\end{matrix}
\right],
\end{equation}
a matrix that is positive-semidefinite (but not positive definite).

{\em Third step.} Since,  \(\widehat {\bar T}\preceq 0\) implies  \( \bar t_{22}\leq 0\), we may write \( 0\geq \bar p_{22}-m/2 = (m/2)({\bar r}^2-1)\), and therefore we have
\[
{\bar r} \leq 1;
\] this imposes a bound \(\lambda \leq \sqrt{m}\) on the convergence rate.

{\em Fourth step.}
We impose the condition \(\bar t_{11}\bar t_{22}-{\bar t}_{12}^2=0\). This results in an equation \(\bar \Xi =0\),
\begin{equation}\label{eq:Xibar}
\bar \Xi({\bar r},\bar b) = {\bar r} b^2 - 2 ({\bar r}^2+1) b+{\bar r}^3+3{\bar r},
\end{equation}
that  relates  \({\bar r}\) (or equivalently the rate \(\lambda\)) and the parameter \(\bar b\) in the differential equation \eqref{eq:Pol_ODE1}.

We observe that the polynomial \(\bar \Xi\) is the limit as \(\delta\downarrow 0\) of the polynomial \(\Xi_\delta\) in \eqref{eq:Xi} (except of course for the symbols used to denote the variables: \(r\) and \(b\) for \(\Xi_\delta\) and \({\bar r}\) and \(\bar b\) for \(\bar \Xi\)). As a consequence, the discontinuous line in Figure~\ref{figure}, presented there as a limit of curves \(\Xi_\delta=0\), also describes the curve \(\bar \Xi =0\) (again after renaming the variables).

The curve of equation \(\bar \Xi({\bar r}, \bar b)=0\) in the \(({\bar r},\bar b)\) plane is invariant with respect to the symmetry \(({\bar r},\bar b)\mapsto (-{\bar r},-\bar b)\) (this is a consequence of the fact that changing \(\bar b\) into \(-\bar b\) in the differential equation is equivalent to reversing the sign of independent variable \(t\)).\footnote{The curves
\(\Xi_\delta(r,b)=0\), \(\delta>0\) do not possess any symmetry because in the discrete algorithm \eqref{eq:nest1}, \(x_{k+1}\) and \(x_{k-1}\) do nor play a symmetric role (or in the terminology of differential equation integrators we are not dealing with time-symmetric algorithms).}  The formula for the roots of a quadratic equation gives
\[
\bar b_{\pm} = \frac{1+{\bar r}^2\pm\sqrt{1-{\bar r}^2}}{{\bar r}}.
\]
From here one may prove that to each real \(\bar b\) there corresponds a unique  \({\bar r}\) such that
\(\bar \Xi({\bar r}, \bar b)=0\).
The maximum value  \({\bar r} =1\) (\(\lambda =\sqrt{m}\)) is achieved only for \(\bar b =2\) (i.e.\ for Polyak's \eqref{eq:Pol_ODE}) and values \({\bar r}\in (0,1)\) correspond to two different real values of  \(\bar b\).

We now have the following result that is proved as in the discrete case.

\begin{theorem}\label{th:2}Consider the differential equation \eqref{eq:Pol_ODE1} (or the equivalent system \eqref{eq:ode1}) with parameter \(\bar b>0\) and assume that \(f\) is \(m\)-strongly convex.
Let \(\lambda = \sqrt{m} {\bar r}\), where \({\bar r}>0\) is the value determined by the relation \(\bar \Xi({\bar r},\bar b) = 0\) (see \eqref{eq:Xibar}) and define the positive semi-definite matrix
\( {\bar P}\) by \eqref{eq:Pcont} and  \eqref{eq:Pcontbis}. Then the matrix \(\bar T\) in \eqref{eq:Tbis} is negative semi-definite.

As a result,  if \(x(t)\) is a solution of \eqref{eq:Pol_ODE1},
the function
\begin{equation}\label{eq:theobis}
\exp(\lambda t)\Big( f(x(t))-f(x_\star)+ [v(t)^{\T},x(t)^{\T}-x_\star^{\T}]\,\bar P\,[v(t)^{\T},x(t)^{\T}-x_\star^{\T}]^{\T}\Big)
\end{equation}
decreases monotonically as \(t\) increases, which implies
 \[f(x(t))-f(x_\star)\leq \bar C \exp(-\lambda t)\] with
 \[
 \bar C = f(x(0))-f(x^\star) + \frac{m}{2} \left\| \frac{1}{\sqrt{m}} \dot x(0)+{\bar r} (x(0)-x^\star)\right\|^2.
 \]
\end{theorem}

\begin{remark} For \(\bar b = 0\), the construction leading to the theorem yields \(r=0\), i.e.\ \(\lambda=0\), and,
\[(\xi(t)-\xi_\star)^{\T}\bar P (\xi(t)-\xi_\star) = \frac{m}{2} \|v\|^2.
\] In addition, \(\bar T = 0\) and
therefore the factor in round brackets in \eqref{eq:theobis} is an invariant of motion.
 In this case the system
\eqref{eq:ode1} is Hamiltonian and the invariant we have found equals \(\sqrt{m}\) times the corresponding Hamiltonian function.
\end{remark}

{\color{black} \begin{remark} The value \(\bar b = 2\), in addition to maximizing the decay rate in \(f(x(t))\) in Theorem~\ref{th:2} for arbitrary \(m\)-strongly convex \(f\), has another optimality property in
 the simple one-dimensional case with \(f(x) = mx^2/2\), when \eqref{eq:Pol_ODE1} or \eqref{eq:ode1} describe a damped harmonic oscillator. An elementary computation (see e.g.\ \cite{LR18}) shows that \(\bar b = 2\) is the value of the friction coefficient that ensures the \emph{fastest dissipation of the energy} \( (\dot x)^2/2+  m x^2/2\).

 It will be proved in the Appendix that if \(f\), in addition to being strongly convex has Lipschitz continuous gradient, then better decay rates in \(f(x(t))\) may be obtained by choosing \(\bar b\) to be  larger than \(2\). Therefore \( (\dot x)^2/2+  m x^2/2\) is not the best Lyapunov function to study the rate of decay of \(f(x)\) in the damped harmonic oscillator. This is in agreement with Theorem~\ref{th:optimal} below.
\end{remark}}

{\color{black}Reference \cite{PS17} gives a Lyapunov function for \eqref{eq:Pol_ODE1} or \eqref{eq:ode1} that includes a cross-term \(v^T\nabla f(x)\) and does not require  the strong convexity of $f$. However, the presence of the gradient in the Lyapunov function makes it necessary that $f$ be demanded to be twice-differentiable (the Hessian of \(f\) appears when differentiating the Lyapunov function with respect to \(t\)).}

\subsection{Optimality}

Steps 2 and 4 in the construction above imply a degree of arbitrariness and it is of interest to ask whether there are alternative choices of \(\lambda\) and \(\widehat{\bar P}\succeq 0\) that, while ensuring \(\widehat{\bar T}\preceq 0\),  furnish better decay rates. We conclude this section by proving that this is not the case.

In the theorem below we use the notation \( \bar r^\star\) and \(\widehat{\bar P}^\star\) for  the values obtained, for given \(\bar b>0\), in the construction leading to Theorem~\ref{th:2}. (These are functions
\( \bar r^\star = \bar r^\star(b)\) and \(\widehat{\bar P}^\star
= \widehat{\bar P}^\star(b)\), but the dependence on \(\bar b\) will be dropped from the notation.) In particular, \(\bar p_{22}^\star ={m\bar r^\star}^2/2\) and \(\bar \Xi(\bar r^\star,\bar b)= 0\).
  The symbols
\(\lambda\) and \(\widehat{\bar P}\) are used in the theorem to refer to an arbitrary real number and an arbitrary \(2\times 2\) symmetric matrix. Finally, we set   \(\lambda^\star=\sqrt{m}\:\bar r^\star\)
 and  \(\lambda =\sqrt{m}\: \bar r\).

\begin{theorem}\label{th:optimal} With the notation as described, for each fixed \(\bar b >0\),  \(\lambda^\star = {\rm max}\: \lambda\), subject to the constraints \(\widehat{\bar T}(\lambda, \widehat{\bar P})\preceq 0\),  \(\widehat{\bar P}\succeq 0\).
\end{theorem}
\begin{proof} Since we are solving a convex optimization problem, it is sufficient to show that \((\lambda^\star, \widehat{\bar P}^\star)\) provides a \emph{local} maximum.

We observed in step 1 above that \modk{\(\widehat{\bar T}\preceq 0\)} determines the values of \(\bar p_{11}\), \(\bar p_{12}\)
as in \eqref{eq:p11p12}. This leaves us with \(\lambda\) (or equivalently \(\bar r\)) and \(\bar p_{22}\) as decision variables. For simplicity we hereafter omit the subindices in \(\bar p_{22}\).

The constraint \(\widehat{\bar P}\succeq 0\), implies \({\rm det} (\widehat{\bar P})\geq 0\) or (after using the values of
\(\bar p_{11}\), \(\bar p_{12}\)) \(\bar p\geq (m/2) {\bar r}^2\).
The constraint \(\widehat{\bar T}\preceq 0\) implies \(\bar t_{11}\bar t_{22}-{\bar t_{12}}^2\geq 0\). We  use \eqref{eq:p11p12}, to write \(\bar t_{11}\bar t_{22}-\bar t_{12}^2\geq 0\) as a function \(\Delta(\bar r,\bar p)\); tedious algebra leads to the expression:
\[
\Delta(\bar r,\bar p) = -\frac{m^3}{2} {\bar r}^4+\frac{\bar b m^3}{2} {\bar r}^3+\left(\frac {m^2\bar p}{2} -\frac{3m^3+\bar b^2m^3}{4}\right) {\bar r}^2+\frac{bm^3}{2}{\bar r}-m{\bar p}^2.
\]

 We will be done if we prove that the pair \((\bar r^\star,\bar p^\star)\) is a local maximum for the problem
\[
{\rm max}\: \bar r \quad {\rm subject\: to}\quad  \bar p- m {\bar r}^2/2\geq 0, \:\: \Delta(\bar r,\bar p)\geq 0.
\]

At the point \(({\bar r}^\star,{\bar p}^\star)\) both constraints are active (in fact they were chosen to be so at steps 2 and 4).
 If we define the Lagrangian
\[\mathcal{L}(\bar r,\bar p) = \bar r +\zeta_1\:(\bar p- m {\bar r}^2/2)+\zeta_2\:\Delta(\bar r,\bar p),
\]
 where \(\zeta_1\), \(\zeta_2\) are the multipliers, the proof concludes by showing that
 the gradient of  \(\mathcal L\) at \(({\bar r}^\star, {\bar p}^\star)\) may be annihilated for a suitable choice of
 \emph{positive} multipliers.

 We impose the requirements
 \[
0= \left. \frac{\partial}{\partial  \bar r}\mathcal{L}\right|^\star =  1 -\zeta_1 m {\bar r}^\star
+\zeta_2
\left. \frac{\partial}{\partial \bar r}\Delta\right|^\star,
\]
(\(|^\star\) means evaluation at at \(({\bar r}^\star, {\bar p}^\star)\))
and
\[
0= \left. \frac{\partial}{\partial \bar p}\mathcal{L}\right|^\star = \zeta_1+\zeta_2 \left(\frac{m^2}{2}
{\bar r}^\star{}^2-2m{\bar p}^\star\right)= \zeta_1 -\zeta_2\frac{m^2}{2}{\bar r}^\star{}^2,
\]
(which implies that \(\zeta_1\) and \(\zeta_2\) have the same sign)
and eliminate \(\zeta_1\) to get
\[
1+\zeta_2 \left(\frac{m^3}{2} {\bar r}^\star{}^3+\left. \frac{\partial}{\partial \bar r}\Delta\right|^\star\right) = 0.
\]
In this way we are left with the task of proving that
\[
\frac{m^3}{2} {\bar r}^\star{}^3+\left. \frac{\partial}{\partial \bar r}\Delta\right|^\star<0,
\]
or, after using the expression for \(\Delta\) and some simplification,
\[
-2{\bar r}^\star{}^3+3\bar b {\bar r}^\star{}^2-(3+{\bar b}^2) {\bar r}^\star+\bar b < 0.
\]
Let us denote by \(\Lambda=\Lambda({\bar r}^\star,\bar b)\) the left hand-side of this inequality.
When \(\bar b = 2\) and \({\bar r}^\star=1\),
we have \(\Lambda = -1\).  On the other hand, we know that
\[
\bar \Xi = {\bar b}^2 \bar r-2 ( {\bar r}^\star{}^2 +1) \bar b+ {\bar r}^\star{}^3+3 {\bar r}^\star  = 0,
\]
and this relation makes it impossible for  \(\Lambda\) to change sign as \(\bar b>0\) and the corresponding \({\bar r}^\star(b)\in(0,1]\) vary.
In fact, if \(\Lambda\) were to vanish, we would have
\[
 \Lambda +\bar \Xi = \big ({\bar r}^\star{}^2 -1\big)\bar b-{\bar r}^\star{}^3=0,
\]
something that cannot happen because \({\bar r}^\star<1\) for \(\bar b \neq 2\).
\end{proof}

\section{Connecting the differential equations with optimization algorithms} \label{sec:connect}

The second-order differential equation \eqref{eq:Pol_ODE1} provides a limit for the algorithm \eqref{eq:nest1} when \(\beta\) changes smoothly with \(h=\sqrt{\alpha}\) in such a way that \(\beta_h = 1- \bar b \sqrt{m} h+o(h)\) as \(h\downarrow 0\).  In this section we study this limit when \(\bar b >0\). As in
\eqref{eq:betarhosq} write \(\beta_h=1-b_h\delta=1-b_h \sqrt{m} h\). Clearly, \(b_h\rightarrow \bar b\) and, in addition, for \(h\) sufficiently small \(b_h\in (b_{\rm min}^h, b_{\rm max}^h)\) (see \eqref{eq:limits}). The application of Theorem~\ref{th:1} then gives a rate \(\rho^2_h=1-r_h \delta = 1-r_h \sqrt{m} h\).
As noted before, the polynomial \(\bar \Xi\) in \eqref{eq:Xibar} is the limit of \(\Xi_\delta\) in \eqref{eq:Xi} as \(h\) (or \(\delta\)) approaches zero, and, accordingly,
\(r_h \rightarrow {\bar r}\),
where \({\bar r}\) solves \(\bar\Xi({\bar r},\bar b)=0\).  Then Theorem~\ref{th:1} guarantees that, over one step \(k\mapsto k+1\) of the algorithm, \(f(x_k)-f(x^\star)\)
decays by a factor \(\rho^2_h = 1-\sqrt{m}{\bar r} h+{ o}(h)\). Over \(k\) steps the decay factor  will be
\((1-\sqrt{m}{\bar r} h+{o}(h))^k\), a quantity that in the limit \(kh\rightarrow t\) converges to
\(\exp(-\sqrt{m}{\bar r} t) = \exp(-\lambda t)\). This is exactly the decay guaranteed by Theorem~\ref{th:2} for \(f(x(t))-f(x^\star)\) over an interval of length \(t\).

In addition, the matrices \(P_h\) in the discrete Lyapunov function converge to the matrix \(\widehat P\) in the differential equation, because from the expression for the entries in \eqref{eq:Pbis} and \eqref{eq:Pcontbis}
\[
p_{11}^h \rightarrow \bar p_{11}, \qquad
p_{12}^h \rightarrow \bar p_{12}, \qquad
p_{22}^h \rightarrow \bar p_{22}.
\]

The above discussion and standard results on the convergence of discretizations of ordinary differential equations imply the following result.

\begin{theorem} Fix the parameter \(\bar b >0\) and the initial conditions \(x(0)\), \(\dot x(0)\)
for  the differential equation \eqref{eq:Pol_ODE1}. For small \(h>0\), consider the optimization algorithm \eqref{eq:nest1} with parameters  \(\alpha= h^2\) and \modk{\(\beta=\beta_h= 1-\bar b \sqrt{m} h+o(h)\)}. Assume that the initial points \(x_{-1}\), \(x_0\)  are such that, as \(h\downarrow 0\),
\(x_0\rightarrow x(0)\) and \((1/h)(x_0-x_{-1}) \rightarrow \dot x(0)\). Then, in the limit \(kh\rightarrow t\),
\begin{enumerate}
\item \(x_k\rightarrow x(t)\) and \((1/h)(x_{k+1}-x_k) \rightarrow \dot x(t)\).
\item The discrete Lyapunov function in \eqref{eq:theo} converges to the Lyapunov function in \eqref{eq:theobis}.
\end{enumerate}
\end{theorem}

\begin{remark}\label{rem:modified}
As a consequende of this theorem, the Lyapunov function of the differential equation
could have been derived alternatively by first finding the Lyapunov
function for the discrete optimization algorithm and then taking
limits. In our research we first investigated the discrete case and then studied the differential equations;
 in hindsight we saw it would have been easier to first deal with the differential equation and then carry out the analysis of the algorithm by mimicking the treatment of the continuous case. {\color{black} References \cite{SDJ18,SDS19,LO20} find Lyapunov functions for different optimization algorithms by first constructing Lyapunov functions for suitable so-called high-resolution differential equations. In our context, this would mean perturbing \eqref{eq:ode1} with suitable \(h\)-dependent terms so as to obtain an (\(h\)-dependent) differential equation for which the algorithm has a high order of consistency. The idea behind those high-resolution equations is very old in the numerical analysis of ordinary and partial differential equations, where they are known as \emph{modified equations}, see e.g. \cite{GSS86} or \cite[Chapter 10]{sanz2018numerical} and, for the stochastic case, \cite{KZ11}.
 }
 \end{remark}

\section{Heavy Ball and other methods}\label{sec:heavy}

{\color{black} The paper \cite{SBC16} has given rise to a number of contributions that aim to understand the behaviour of optimization methods by seeing them as discretizations of differential equations. However}  it is well known that the long-time properties of a differential equation are not automatically inherited by their discretizations, regardless of the value of the step-size chosen.
A very simple example is provided by the application of Euler's rule to the harmonic oscillator: for all step-sizes the discrete trajectories grow while the continuous solutions stay bounded. A more relevant example in an optimization context may be seen in \cite{SSKZ20}. {\color{black} On the other hand properties of the discretizations may often be extrapolated to the continuous limit; a general discussion of these points in different settings may be seen in \cite{A19}.}

In the setting of the preceding section,
it is not true that discretizing a dissipative differential equation with a known  a Lyapunov function will always yield an optimization algorithm with a \lq\lq suitable\rq\rq\ Lyapunov function. We now illustrate this fact by means of the Heavy Ball algorithm
obtained by choosing \(\gamma = 0\) and \(\beta \neq 0\) in \eqref{eq:gen_eq}.

 We proceed as in Section~\ref{sec:nest_anal}, rewrite the algorithm in terms of \(d_k\) and \(x_k\) and then cast it in the general format \eqref{eq:control_disc}.
We will {\color{black} presently} prove that  a discrete Lyapunov \emph{with properties similar to the Lyapunov function for Nesterov's method in Theorem~\ref{th:1} does not exist}.
We argue by contradiction. With the notation as in Section~\ref{sec:nest_anal}, we consider
\begin{itemize}
 \item  \(p_{ij}=m\,\phi_{ij}(\beta,\delta)\), \((i,j)= (1,1), (1,2), (2,2)\), such that \(\widehat P\succeq 0\),
 \item \(r=\psi(\beta,\delta)>0\),
 \item \(c>0\),
\end{itemize}
 and suppose that the corresponding \(T(\lambda,P)\) is \(\preceq 0\) for each \(\delta<c/\sqrt{\kappa}\). As in Remark~\ref{rem:dim2} to ensure equivariance with respect to changes of scale, the number \(c\) and functions \(\phi_{ij}\) and \(\psi\) are assumed to be independent of the constants \(m\) and \(L\) associated with \(f\) and the values of the parameters \(\alpha\) and \(\beta\) in the Heavy Ball algorithm.

For future reference, the element \(t_{11}\) is found to have the expression:
 \[t_{11} = (\beta^2-\rho^2)p_{11}+2\delta \beta^2p_{12}+\delta^2\beta^2p_{22}+\delta^2(L-m)\beta^2/2.
 \]
 This has to be \(\leq 0\) for \(\delta<c/\sqrt{\kappa}\).

 Next, as in the preceding section, we assume that \(\beta\) changes smoothly with \(h\) in such a way that, for some \(\bar b>0\),    \(\beta=\beta_h=1-\bar b\delta+o(h)=1-\bar b \sqrt{m} h+o(h)\). Clearly the algorithm is then a consistent discretization of the differential equation \eqref{eq:Pol_ODE1}, and we assume that \(r_h\), \(p_{ij}^h\) converge to their differential equation counterparts \(\bar r\) and \({\bar p}_{ij}\).\footnote{This hypothesis is not necessarily in the argument that follows. It is enough to suppose that
 \(r_h\), \(p_{ij}^h\) have finite limits. }

 In this situation:
 \[
 0\geq \delta^{-1} t_{11}^h = \frac{\beta_h^2-\rho_h^2}{\delta} p_{11}^h+2 \beta_h^2 p_{12}^h+\delta\beta_h^2p_{22}^h+\frac{c}{2}\,\sqrt{\frac{m}{L}}(L-m)\beta_h^2,
 \]
 and, taking limits,
 \begin{equation}\label{eq:contradiction}
 0\geq -2\frac{\bar b-\lambda}{\sqrt{m}} {\bar p}_{11}+2{\bar p}_{12}+\frac{c}{2}\,\sqrt{\frac{m}{L}}(L-m).
 \end{equation}
 This cannot happen because \(L\) may be arbitrarily large.

 {\color{black}\begin{remark}The Heavy Ball algorithm is a \lq\lq more natural\rq\rq\ discretization of \eqref{eq:Pol_ODE1} than Nesterov's, in that, as conventional linear multistep methods, it does not evaluate \(\nabla f\) at a linear combination of \(x_k\), \(x_{k-1}\) (cf.\ Remark~\ref{re:oneleg}).
 \end{remark}
 }%end of blue color

 {\color{black}\begin{remark} The contradiction in \eqref{eq:contradiction} arises because we insisted in \(T\) being \(\preceq 0\) for \lq\lq large\rq\rq\ non-dimensional stepsizes
 \(\delta= \sqrt{m}h<c/\sqrt{\kappa}\). For optimization algorithms that, in the limit \(h\downarrow 0\), approximate a differential equation with decay \(\exp(-\lambda h)=\exp(-\bar{r}\delta) \) in a time-interval of length \(h\), such large stepsizes seem to be necessary to achieve accelerated rates \(1-\mathcal{O}(\sqrt{\kappa})\) rather than rates \(1-\mathcal{O}(\kappa)\).

 The reference \cite{SDJ18}
 constructs a Lyapunov function for the Heavy Ball method, but it only operates for \(\delta =\mathcal{O}(1/\kappa)\)
 and, while useful in showing convergence, does not provide acceleration. For an additional  convergence proof of the Heavy Ball algorithm see \cite{GFJ15}; again this reference does not prove acceleration.
 \end{remark}
 }

 The three-parameter family of methods \eqref{eq:gen_eq}  contains algorithms, like Nesterov's, that \lq\lq inherit\rq\rq\ the ODE Lyapunov function {\color{black} for  stepsizes \(\delta<c/\sqrt{\kappa}\)} and algorithms, like the Heavy Ball, that do not. In fact the situation for the Heavy Ball is arguably the rule rather than the exception. For  \eqref{eq:gen_eq},
 \[t_{11} = (\beta^2-\rho^2)p_{11}+2\delta \beta^2p_{12}+\delta^2\beta^2p_{22}+\delta^2(L-m)(\beta-\gamma)^2/2-m\gamma^2\delta^2/2;
 \]
 where we observe the unwelcome presence of the factor \(L-m\) that created the difficulties  in the {\color{black}  analysis of the} Heavy Ball algorithm. If we look at a situation where \(\beta\) changes with \(h\) as above and in addition \(\gamma\) is also allowed to change with \(h\) and approaches a limit, a Lyapunov function {\color{black} that has the form envisaged and  works for \(\delta<c/\sqrt{\kappa}\)} may only exist if \(\beta_h-\gamma_h\) vanishes (at least in the limit \(h\downarrow 0\)) to offset the factor, i.e. if the algorithm is not far away from Nesterov's.
 \bigskip 

{\color{black} {\bf Acknowledgement.} We are thankful to an anonymous referee for helping us to improve the discussion of our results.}

\section*{Appendix}
In Theorem~\ref{th:optimal} we proved that, for each \(\bar b>0\), the rate of decay \(\lambda\) provided by Theorem~\ref{th:2} is the best one may obtain by using Theorem~\ref{theo:main_con} \emph{if one chooses} \(\sigma = 0\).
In this Appendix we investigate whether \(\lambda\) may be improved by a suitable choice of \(\sigma >0\). Since for \(\sigma\neq 0\), the matrix \(\bar M^{(3)}\) that contains the constant \(L\) contributes to \(T\), the following results require that \(f\), in addition to being \(m\)-strongly convex (as in Theorem~\ref{th:2}) is \(L\)-smooth, i.e.\ they hold for \(f\in  \mathcal{F}_{m,L}\).

When \(\sigma \neq 0\) the expressions for the \(t_{ij}\) in Section~\ref{sec:ode} have to be replaced by:
 \begin{align*}
  \bar t_{11} &=  -2\bar b \bar p_{11}+2\sqrt{m} \bar p_{12}+\lambda \bar p_{11},\\
  \bar t_{12} &= -\bar b \sqrt{m} \bar p_{12}+\sqrt{m} \bar p_{22}+\lambda \bar p_{12}, \\
  \bar t_{13} &= -(1/\sqrt{m})\bar p_{11}+\sqrt{m}/2, \\
  \bar t_{22} &= \lambda  \bar p_{22}-(m/2)\lambda-\sigma m L/(m+L),\\
  \bar t_{23} &= -(1/\sqrt{m})\bar p_{12}+\lambda/2+\sigma/2,\\
  \bar t_{33}&=  -\sigma/(m+L).
\end{align*}
As  in Section~\ref{sec:ode}, we set \(\lambda = \sqrt{m}\:{\bar r}\) and, in addition, \(\sigma = m \bar s\) (the variable \(\bar s\) is, as \(\bar r\), non-dimensional). We shall show that it is possible, for given \(m\) and \(L\), to find values of the six  parameters \({\bar p}_{11}\), \({\bar p}_{12}\), \({\bar p}_{22}\), \(\bar b\), \(\bar s\),  \(\bar r\), in such a way that the constraints
\(\widehat{\bar T}\preceq 0\),  \(\widehat{\bar P}\succeq 0\), \(\bar s\geq 0\) are satisfied and, at the same time, \(\bar r>1\), so that by using the matrix \({\bar M}^{(3)}\) it is possible to improve on the best value \(\bar r=1\) (associated with \(\bar b =2\) and leading to \(\lambda = \sqrt{m}\)) that may be achieved in Theorem~\ref{th:2}.

For given \(m\) and \(L\), we determine the values of the six parameters as follows:

\emph{First step.} We impose \({\bar t}_{22}=0\), a requirement that leads to the relation
\[
\frac{{\bar p}_{22}}{m} = \frac{1}{2}+\frac{\bar s}{\bar r} \frac{\kappa}{\kappa+1}.
\]

\emph{Second step.} We impose \({\bar t}_{23}=0\) and get
\[
\frac{{\bar p}_{12}}{m} = \frac{\bar r+\bar s}{2}.
\]

\emph{Third step.} We require \({\rm det}(\widehat{\bar P}) = 0\). Therefore
\[
\frac{{\bar p}_{11}}{m} = \frac{({\bar p}_{12}/m)^2}{{\bar p}_{22}/m}.
\]
Note that for \(\bar r, \bar s\geq 0\) we have \({\bar p}_{22}>0\) and thus the third step guarantees that
\(\widehat{\bar P}\succeq 0\).

\emph{Fourth step.} We next demand that \({\bar t}_{12}=0\) and obtain
\[
\bar b = \bar r+\frac{{\bar p}_{22}/m}{{\bar p}_{12}/m}.
\]
The four preceding displayed formulas allow us to express the parameters \({\bar p}_{12}\), \({\bar p}_{22}\), and \(\bar b\) as known functions of \(\bar s\) and  \(\bar r\).

\emph{Fifth step.} At this stage, we have ensublue that \({\bar t}_{12}\), \({\bar t}_{22}\), \({\bar t}_{23}\) vanish. As a result, the condition \(\widehat{\bar T}\preceq 0\) is equivalent to  \(\widehat{\bar T}^{13}\preceq 0\)
where \(\widehat{\bar T}^{13}\) is the \(2\times 2\) matrix obtained by suppressing from \(\widehat{\bar T}\) its second row and column. Furthermore \({\bar t}_{33}<0\) for \(\bar s >0\) and then we shall have \(\widehat{\bar T}^{13}\preceq 0\) if we impose that \({\rm det}(\widehat{\bar T}^{13})=0\), or
\[
{\bar t}_{11}{\bar t}_{33}- {\bar t}_{13}^2 = 0.
\]
By using the displayed formulas above, the last equation becomes a relation
\(F(\bar r,\bar s) = 0\), between \(\bar r\) and \(\bar s\), with
\[
F = \frac{{\bar r}^2{\bar s} (\bar r+\bar s)^2}{2(\kappa+1) \bar r+4\kappa \bar s}
-\frac{1}{4}\left( \frac{(\kappa+1)\bar r(\bar r+\bar s)^2 }{(\kappa+1) \bar r+2\kappa \bar s}-1\right)^2.
\]
We next show that the rational curve  \(F(\bar r,\bar s) = 0\) in the \((\bar r,\bar s)\) real plane has points with \(\bar s >0\) and \(\bar r >1\).

It is easily checked that the point \(\bar r = 1\), \(\bar s = 0\) lies on the curve \(F=0\) and has \(\bar b=0\).
This could have been anticipated because, if \(\bar s =0\) and \(\bar b = 2\), the construction in this appendix just reproduces the construction in Section~\ref{sec:ode}, which  yields \(\bar r = 1\).

By removing the denominator in the rational function \(F\) so as to have a polynomial equation for the curve and looking at the  Newton diagram at \(\bar r = 1\), \(\bar s = 0\), one sees that in the neighbourhood of this point
the curve consists of a single branch that may be parameterized by \(\bar r\). A Taylor expansion reveals that
\[
\bar s = 2(\kappa+1) (\bar r-1)^2 + \mathcal{O}((\bar r-1)^3).
\]
In this way, choosing a sufficiently small value of the parameter \(\bar s>0\), there are two possible values of the rate \(\bar r\)
\[
\bar r\approx 1\pm \sqrt{\frac{\bar s}{2(\kappa +1)}},
\]
one of which is \( >1\).
In conclusion we have proved analytically that the introduction of \(\sigma\) and \(\bar M^{(3)}\) in \(T\) makes it possible to \emph{achieve rates \(\bar r >1\)} (or \(\lambda > \sqrt{m}\)).

\begin{table}
\label{table}

\begin{center}
\begin{tabular} {c|cc|cccc}$ \kappa$ & $\bar b-2$ & $\bar r-1$ & $\bar s$ &
$\frac{{\bar p}_{11}}{m} -\frac{1}{2}$ & $\frac{{\bar p}_{12}}{m} -\frac{1}{2}$& $\frac{{\bar p}_{22}}{m} -\frac{1}{2}$\\\hline
$10^1$ & 3.5(-1)  & 8.6(-2) & 4.1(-1) & 1.6(-1) & 2.5(-1) & 3.4(-1) \\
$10^2$ & 2.2(-1) & 1.8(-2) & 1.3(-1) & 2.7(-2) & 7.6(-2) & 1.3(-1) \\
$10^3$ & 1.0(-1) & 3.9(-3) & 5.5(-2) & 5.2(-3) & 2.9(-2) & 5.5(-2) \\
$10^4$ & 4.7(-2) & 8.2(-4) & 2.4(-2) & 1.1(-3) & 1.3(-2) & 2.4(-2) \\
$10^5$ & 2.1(-2) & 1.8(-4) & 1.1(-2) & 2.3(-4) & 5.5(-3) & 1.1(-2) \\
$10^6$ & 9.9(-3) & 3.8(-5) & 5.0(-3) & 5.0(-5) & 2.5(-3) & 5.0(-3) \\
$10^7$ & 4.6(-3) & 8.1(-6) & 2.3(-3) & 1.1(-5) & 1.2(-3) & 2.3(-3) \\
$10^8$ & 2.2(-3) & 1.7(-6) & 1.1(-3) & 2.3(-6) & 5.4(-4) & 1.1(-3) \\
$10^9$ & 9.9(-4) & 3.8(-7) & 5.0(-4) & 5.0(-7) & 2.5(-4) & 5.0(-4)
\end{tabular}
\end{center}

\caption{Value of the dissipation parameter \(\bar b\) in the differential equation that leads to the best rate of decay \(\bar r\) for different choices of the condition number \(\kappa\). The table also gives the values of the parameters to construct the matrices \(\widehat{\bar T}\preceq 0\),  \(\widehat{\bar P}\succeq 0\).}
\end{table}

We next determined the value of \(\bar s\) that leads to the largest possible \(\bar r\) on the curve \(F=0\). In view of the involved expression of \(F\), we proceeded numerically and found  this largest value by continuation along the curve, starting from \(\bar r = 1\), \(\bar s = 0\).
The results, for different values of \(\kappa\), are given in Table~\ref{table}. For the small condition number \(\kappa=10\), the table shows that it is possible to achieve a
decay \(\approx\exp(-1.086\sqrt{m}t)\) by fixing the dissipation coefficient at the value \(\bar b \approx 2.35\)
rather than at \(\bar b=2\) as in Polyak's \eqref{eq:Pol_ODE}---this is a marginal improvement on the best decay \(\exp(-\sqrt{m}t)\) that one may insure without using \(\bar M^{(3)}\). In addition the improvement quickly decreases as the condition number grows: for \(\kappa = 10^3\) the decay is \(\exp(-1.0039\sqrt{m}t)\). In fact, we observe in the table that, as \(\kappa\uparrow \infty\), \(\bar r\approx 1+ 0.38\kappa^{-2/3}\). Of course as \(\kappa\) increases, \(\bar r\) and \(\bar b\) approach the values \(1\) and \(2\) that correspond to the situation studied in Section~\ref{sec:ode}, where \(f\) is not assumed to possess Lipschitz gradients. A similar convergence obtains for the matrix
 \(\widehat{\bar P}\succeq 0\). Also note that \(\bar s \approx 0.50 \kappa^{-1/3}\): as the condition number increases
the parameter \(\sigma = \sqrt{m}\bar s\) that multiplies \(\bar M^{(3)}\) decreases, as it may have been expected.

The results in the appendix and the connection between discrete and continuous Lyapunov functions strongly suggest that there would have been no substantial gain in the rate \(\rho^2\) found  in Section~\ref{sec:nest_anal} if we had allowed \(\ell\neq 0\) there.
%%%%%%%%%%%%%%%%%%%%

%%%%%%%%%%%%%%%%%%%%%%%%%%%%%%%%%%%%%%%%%%%

\end{document}